\documentclass[a4paper,12pt,leqno]{amsart}
\usepackage{latexsym}

\usepackage[width=5.65in,height=7.85in,centering]{geometry}

\usepackage[colorlinks,plainpages,backref,urlcolor=blue]{hyperref}
\usepackage[all]{xy}

\usepackage{amssymb} % \mathbb
\usepackage{amsmath} % \mathbb
\usepackage{bbm}

\usepackage{tikz}
\usetikzlibrary{arrows,cd}  % good for commutative diagrams
\usepackage{amsthm}

\def\P{{\mathbf{P}}}% \P == \mathbb{P}
\def\Z{{\mathbb{Z}}}% \Z == \mathbb{Z}
\def\K{{\mathbb{K}}}% \K == \mathbb{K}
\def\CC{{\mathbb{C}}}% \C == \mathbb{C}
\def\R{{\mathbb{R}}}% \R == \mathbb{R}
\def\A{{\mathcal{A}}}% \A == \mathcal{A}
\def\B{{\mathcal{B}}}% \B == \mathcal{B}
\def\oS{{\overline{S}}}  % coordinate ring of a restriction

\def\one{{\mathbbm1}}   % vector of all ones

\def\pp{{\mathfrak p}}  % a prime ideal
\newcommand{\abs}[1]{\left|#1\right|}     % absolute value delimiters
\newcommand{\set}[1]{\left\{#1\right\}}   % parentheses around a set
\DeclareMathOperator{\rank}{rank}
\DeclareMathOperator{\codim}{codim}

\DeclareMathOperator{\Der}{Der}

\DeclareMathOperator{\pd}{pd}

\DeclareMathOperator{\spec}{Spec}
\DeclareMathOperator{\res}{res}

\DeclareMathOperator{\POexp}{POexp}

\DeclareMathOperator{\Proj}{Proj}

\DeclareMathOperator{\Ext}{Ext}
\numberwithin{equation}{section}

\newtheorem{theorem}{Theorem}[section]
\newtheorem{prop}[theorem]{Proposition}
\newtheorem{cor}[theorem]{Corollary}
\newtheorem{lemma}[theorem]{Lemma}

\newtheorem{conj}[theorem]{Conjecture}
\theoremstyle{remark}
\newtheorem{rem}[theorem]{Remark}
\newtheorem{example}[theorem]{Example}
\theoremstyle{definition}
\newtheorem{define}[theorem]{Definition}

\title{On Ziegler's conjectures for logarithmic derivations of arrangements}
\author{Takuro Abe and 
Graham Denham}% (\versionnumber)}
\date{\today}

\begin{document}

\maketitle

\begin{abstract}
  In his paper and thesis in 1989, Ziegler posed several conjectures regarding
  commutative algebra related to hyperplane arrangements. 
In this article, we revisit two of them.  One is on generic cuts of free arrangements, and the other has to do with minimal degree generators for the logarithmic differential forms. We prove the first one, and disprove the second one.
We also give some positive answers to related problems he posed, using recent
developments in arrangement theory.
\end{abstract}

\section{Introduction}
Let $\K$ be a field, $V=\K^\ell$, and
$S=\K[x_1,\ldots,x_\ell]$ its coordinate ring. Let $\Der S:=\oplus_{i=1}^\ell S \partial_{x_i}$ and 
$\Omega^p_V:=
\Lambda^p\oplus_{i=1}^\ell S dx_i$. 
Let $\A$ be a central, essential arrangement of hyperplanes in $V$, i.e., a finite set of hyperplanes in $V$ whose intersection is $0$. For 
each $H \in \A$ fix 
$\alpha_H \in V^*$ such that 
$\ker \alpha_H=H$, and 
let $Q(\A):=
\prod_{H \in \A} \alpha_H$. Then the module of \textbf{logarithmic vector fields}
on $V$ with respect to $\A$ is defined by 
\[
D(\A):=\{\theta \in \Der S \mid \theta(\alpha_H) \in S\alpha_H\ (\forall H \in \A)\},
\]
and \textbf{the module of logarithmic differential $1$-forms} by
\[
\Omega^1(\A):=
\{
\omega \in \displaystyle \frac{1}{Q(\A)} \Omega^1_V 
\mid 
\omega \wedge d\alpha_H \in 
\displaystyle \frac{\alpha_H}{Q(\A)} \Omega^2_V \ 
(\forall H \in \A)\}.
\]
These modules are mutually dual (see \cite{OT} for example), thus both are 
reflexive $S$-modules of rank $\ell$; however, they are not free in general.
They inherit the standard grading from $S$ by letting $\deg\partial_{x_i}=0$
for each $i$.  We say that $\A$ is \textbf{free} with $\exp(\A)=(d_1,d_2,\ldots,d_\ell)$ if 
\[
D(\A) \cong \oplus_{i=1}^\ell S[-d_i].
\]
We may assume $1=d_1\leq d_2\leq\cdots\leq d_\ell$.
Taking duals, this is equivalent to
\[
\Omega^1(\A) \cong \oplus_{i=1}^\ell S [d_i].
\]

The homological properties of these modules have been a subject of interest
since their inception in the paper of Kyoji Saito~\cite{Sa}.  For hyperplane
arrangements, the interaction of properties like freeness or degrees of
generation with the combinatorics of the arrangements has proven to be
especially subtle and is the subject of extensive literature: see, for
example, \cite{A2,A4,A5,A6}.

Outside of the free case, $D(\A)$ and $\Omega^1(\A)$ have similar but
inequivalent behaviour with respect to the operation of deleting a hyperplane.
For example, in \cite{A5} it is shown that if $\A$ is free, then for any
$H \in \A$, the projective dimension 
$\pd_S D(\A \setminus \{H\})$ is at most one. On the other hand, even if $\A$ 
is free, 
there are examples for which $\pd_S \Omega^1(\A \setminus \{H\})=\ell-2$ when 
$\ell=4$
(noting that this is as large as possible for a reflexive module, 
see non-vanishing of $\mbox{Ext}^2$ in 
\cite[Example\ 5.10]{AD}).

Deletion-contraction arguments for matroids and arrangements are a fundamental
tool.
For a hyperplane $H\in\A$, let $\A':=\A\setminus\set{H}$ denote the deletion and
$\A^H:=\set{H\cap K\mid K\in \A'}$ the restriction.
Logarithmic derivations fit into the following \textbf{Euler exact sequence}:
\begin{equation}\label{eq:Dseq}
0 \rightarrow 
D(\A') 
\stackrel{\cdot \alpha_H}{\rightarrow}
D(\A)
\stackrel{\rho^H}{\rightarrow}
D(\A^H).
\end{equation}
The map $\rho^H$ is induced by the canonical surjection $S\to\oS:=S/\alpha_H S$:
for $\theta\in D(\A)$ and $\overline{f}\in \oS$, we let
\[
\rho^H(\theta)(\overline{f}):=
\overline{\theta(f)}.
\]
The map $\rho^H$ is called the \textbf{Euler restriction map}.
The exact sequence \eqref{eq:Dseq} links algebraic and combinatorial properties
of $\A$: in particular, it is central to the proof of Terao's famous
Addition-Deletion Theorem~\cite{T1}.

The analogous one for logarithmic forms was first considered by
Ziegler~\cite{Z2}:
\[
0 \rightarrow \Omega^1(\A) \stackrel{\cdot \alpha_H}{\rightarrow}\Omega^1(\A') 
\stackrel{\res_H}{\rightarrow} \Omega^1(\A^H),
\]
who showed that logarithmic forms were suitably functorial on arrangements (see
\cite[Thm.\ 4.4]{Z2}.  Here, $\res_H$ denotes the pullback of differential forms along the inclusion $H\hookrightarrow V$.  
This is to say that, if $\omega$ is a logarithmic form on $\A'$ whose
restriction to $H$ is zero, then $\alpha_H^{-1}\omega$ is logarithmic on
$\A$.
We note that the restriction map of forms was considered from arrangements to
multiarrangements
by Ziegler~\cite{Z} and used to advantage by Yoshinaga in \cite{Y1}; however,
it seems to have been overlooked for arrangements without multiplicities.

In \cite{Z2}, Ziegler gave several conjectures related to algebraic
structure of $\Omega^1(\A)$. We revisit Ziegler's conjectures in this paper.
For that purpose, let us give a precise defintion of a
``generic'' subspace.

\begin{define}
  For an essential arrangment $\A$ in $V=\K^\ell$, let
\[
L(\A):=\set{\cap_{H \in \B} H \mid \B \subseteq \A}
\]
be the \textbf{intersection lattice} of $\A$, ordered by reverse
inclusion and ranked by codimension.  Let $X$ be a subspace not in $L(\A)$.
  
We say that $X$ is \textbf{$k$-generic} with respect to $\A$ if 
$\codim_V (X \cap Y)=\codim X+\codim Y$ for all $Y \in L_{\le k}(\A)$. 
We say that $X$ is 
\textbf{generic} if $X$ is $(\dim X)$-generic. 
\label{kgeneric}
\end{define}

For example, a hyperplane $H\subset V$
is generic with respect to $\A$ if $\codim (X \cap H)=\codim X+1$ for 
all $X \in L(\A\setminus\set{H}) \setminus \{0\}$. A hyperplane $H$ is $1$-generic if and
only if $H\not\in\A$.  Now let us recall the first conjecture we consider:

\begin{conj}[\cite{Z2}, Conjecture 7.3]
Let 
$\A$ be an  essential and free arrangement in $\K^\ell$ with $\ell \ge 4$. 
Let $X$ be a generic subspace with $\dim X >2$, and 
let $\A \cap X$ denote the restriction of $\A$ onto $X$: that is,
\[
\A \cap X:=\{ H \cap X \mid H \in \A' \}
\]
Then the restriction map 
\[
\res_X\colon \Omega^1(\A) \rightarrow \Omega^1(\A \cap X)
\]
is surjective, and the number of the minimal set of generators for 
$\Omega^1(\A)$ and $\Omega^1(\A \cap X)$ are the same.
\label{Zconj}
\end{conj}

Ziegler wrote in \cite{Z2} that if Conjecture \ref{Zconj} were true,
then we would have the following implication,
which was subsequently proved by Yuzvinsky \cite{Yu}.

\begin{theorem}[\cite{Z2}, Corollary 3.5, \cite{Yu}]
  For $\A$ and $X$ as in Conjecture \ref{Zconj}, the arrangement
  $\A \cap X$ is never free.  In other words, 
a generic cut of a free arrangement is not free.
\label{Zconj3}
\end{theorem}

To state Ziegler's second conjecture, let us 
introduce the following definition:

\begin{define}
  For an integer $k$ we say
$\A$ is \textbf{$k$-critical} if 
$\Omega^1(\A)_{-k} \neq 0$, and 
$\Omega^1(\A \setminus \{H\})_{-k} = 0$ 
for any $H \in \A$.
\label{critical}
\end{define}

If $\A$ is $k$-critical for some $k\geq0$, then
$\abs{\A}-\abs{\A^H} \ge k$ for all $H \in \A$, which motivated the
following:

\begin{conj}[\cite{Z2}, Conjecture 8.6]
  If $\A$ is $k$-critical, there exists some $H\in\A$ for
  which $\abs{\A}-\abs{\A^H} = k$.
\label{Zconj2}
\end{conj}

The purpose of this article is to prove Conjecture \ref{Zconj} and by it we give another proof 
of Theorem~\ref{Zconj3}. Also we disprove Conjecture \ref{Zconj2}. 
Our Theorem~\ref{zconjverygeneral} gives a somewhat more general
version of Conjecture \ref{Zconj} and Theorem \ref{Zconj3}, using
recent bounds on the regularity of $\Omega^1(\A)$ due to Dan Bath~\cite{Bath22}.

\section*{Acknowledgements}
The first author is
partially supported by JSPS KAKENHI Grant Numbers JP18KK0389 and JP21H00975.
The second author is supported by NSERC of Canada, and would like to thank
the Max Planck Institute for Mathematics for its hospitality during the final
preparation of this manuscript.
The authors are grateful to Avi Steiner for his careful reading of an earlier
draft, as well as to an anonymous referee.

\section{Preliminaries}
Let us recall some relevant results about arrangements.  By convention, we
assume that hyperplane arrangements are \textbf{simple}: that is, the
linear forms $\set{\alpha_H\mid H\in\A}$ are pairwise linearly independent.
In order to keep track of multiplicities, for each $m\in\Z_{>0}^\A$, let
\[
Q(\A,m):=\prod_{H\in\A}\alpha_H^{m(H)}.
\]
The pair $(\A,m)$ is a \textbf{multiarrangement} and can also be regarded as a
linear realization of a matroid without loops in which (unordered)
parallel edges are allowed.
Also recall the higher order logarithmic modules, defined as follows:

\begin{eqnarray*}
D^p(\A,m)&:=&\set{ \theta \in \wedge^p \Der S \mid 
\theta(\alpha_H,f_2,\ldots,f_p) \in S \alpha_H^{m(H)}\ (\forall H \in \A,\ 
\forall f_2,\ldots,f_p \in S)},\\
\Omega^p(\A,m)&:=&\set{ \omega \in \frac{1}{Q(\A,m)}\Omega^p_V \mid 
(Q(\A,m)/\alpha_H^{m(H)}) \omega \wedge d\alpha_H \in \Omega^1_V\ (\forall H \in \A)}.
\end{eqnarray*}

We will make use of the 
``Strong Preparation Lemma'' of Ziegler~\cite[Lem.\ 5.1]{Z2}, reformulated
as in \cite{AD}.
\begin{lemma}
Let $x_1,x_2,\ldots,x_\ell$ be coordinates for $V^*$ and assume $\alpha_H=x_1$.
  Let 
$\omega=\sum_{i=1}^\ell \displaystyle \frac{f_i}{Q}dx_i 
\in \Omega^1(\A)$, where $Q:=Q(\A)$ and $f_i
\in S$. Then 
\[
f_1 \in (\alpha_H,\prod_{X \in \A^H} \alpha_X).
\]
Here 
$\alpha_X$ denotes $\alpha_L$ for a chosen $L \in \A$ with $L \cap H=X$. 
If we let $B:=\prod_{X \in \A^H} \alpha_X$, then $\deg B/Q=-|\A|+|\A^H|$.

%Let $x_1,x_2,\ldots,x_\ell$ be coordinates for $V^*$ and assume %$\alpha_H=x_1$.
%  Let 
%$\omega=\sum_{i=1}^\ell \displaystyle \frac{f_i}{Q}dx_i 
%\in \Omega^1(\A,m)$, where $Q:=Q(\A,m)$ and $f_i
%\in S$. Then 
%\[
%f_1 \in (\alpha_H,\prod_{X \in \A^H} \alpha_X^{m^*(H)}).
%\]
%Here 
%$\alpha_X$ denotes $\alpha_L$ for a chosen $L \in \A$ with $L \cap %H=X$. Thus 
%$d_X+e_X=|m_X|$.
%
%If we denote $B:=\prod_{X \in \A^H} \alpha_X^{e_X}$, then $\deg B/Q=-%|m|+|m^*|$.
\label{Tlemma}
\end{lemma}

\begin{theorem}[Terao's Addition-Deletion theorem, \cite{T1}]
Let $ H \in \A,\ \A':=\A \setminus \{H\}$ and let $\A^H:=\{H \cap L \mid 
L \in \A'\}$. Then any two of the following three implies the third:
\begin{itemize}
    \item [(1)]
    $\A$ is free with $\exp(\A)=(d_1,\ldots,d_{\ell-1},d_\ell).$
    \item [(2)]
    $\A'$ is free with $\exp(\A')=(d_1,\ldots,d_{\ell-1},d_\ell-1).$
    \item [(3)]
    $\A^H$ is free with $\exp(\A^H)=(d_1,\ldots,d_{\ell-1}).$
\end{itemize}
In particular, all the three above hold if $\A $ and $\A'$ are free.
\label{adddel}
\end{theorem}

\begin{theorem}[{\cite[Theorem 1.13]{A9}}] 
If $\A':=\A \setminus \{H\}$ is free, then
\[
\rho^H\colon D^p(\A) \rightarrow D^p(\A^H)
\]
is surjective for all $p \in \Z_{\ge 0}$.
\label{thm:FST}
\end{theorem}

\begin{theorem}[{\cite[Theorem\ 1.3]{AD}}]
If $\A$ is free, then 
  \[
  \res_H\colon \Omega^1(\A')\rightarrow \Omega^1(\A^H)
  \]
is surjective for all $H \in \A$.
\label{thm:FST2}
\end{theorem}

In order to describe the structure of modules of logarithmic differential forms
which are close to being free, let us introduce the following.

\begin{define}[\cite{A5}, \cite{AD}]
An $S$-graded module $M$ of rank $\ell$ is \textbf{
  strongly plus-one generated (SPOG)} if there is a minimal free resolution of the following form:
\[
\begin{tikzcd}[column sep=small]
  0 \ar[r] & S[-d-1] \arrow{r}{\text{$(f_1,\ldots,f_\ell,\alpha)$}} &[3.5em]
  \oplus_{i=1}^\ell S[-d_i] \oplus S[-d] \ar[r] & M \ar[r] & 0,\\
\end{tikzcd}
\]
where $d,d_i \in \Z$ and $0 \neq \alpha \in V^*$.
The sequence of integers $\POexp(M):=(d_1,\ldots,d_\ell)$, written
in non-decreasing order, is called the \textbf{exponent vector} of 
the 
SPOG module $M$.
The integer $d$ is called the \textbf{level} of the SPOG module $M$. The 
corresponding degree $d$-element $\varphi\in M$ is called a
\textbf{level element}. Also, a minimal generating set of $M$ is called 
an \textbf{SPOG generator} for $M$.
\label{SPOGdef} 
\end{define}

\begin{define}[\cite{AD}]
We say that $\A$ is \textbf{dual SPOG} if the module 
$\Omega^1(\A)$ is SPOG. 
\label{dualSPOG}
\end{define}

\begin{theorem}[\cite{AD}]
  Suppose that $\A$ is free. Then $\A^+:=\A\cup\set{H}$ is either free or
  dual SPOG with 
$\POexp(\Omega^1(\A^+))=\exp(\A)$ and level equal to $-\abs{\A^+}+\abs{{\A^{+H}}}$. 
\label{SPOGmain}
\end{theorem}

For a graded $S$-module $M$, let $\widetilde{M}$ denote the corresponding
sheaf on $\P^{\ell-1}=\Proj(S)$.  Since modules of logarithmic forms
are reflexive, we have
\begin{prop}[Prop.\ 2.5, \cite{AD}]
\[
H_*^0(\widetilde{\Omega^1(\A)}):=
\bigoplus_{j \in 
\Z} H^0(\widetilde{\Omega^1(\A)}(j))=
\Omega^1(\A).
\]
\label{gs}
\end{prop}

%Let $\chi(\A;t):=\sum_{X \in L(\A)} \mu(X) t^{\dim X}=\sum_{i=0}^\ell 
%b_i(\A) t^{\ell-i}(-1)^i$ be the 
%\textbf{characteristic polynomial} of $\A$. 
\begin{theorem}[\cite{A}]
Let $\A$ be a free arrangement in $\K^3$ with $\exp(\A)=(1,a,b)_\le$, i.e., $1 \le a \le b$. Then 
for any $H \in \A$, it holds that $|\A^H|\le a+1$ or $|\A^H|=b+1$.
\label{A}\end{theorem}

The following result is well-known when $m=\one$.  For a proof in general,
see \cite[Lemma\ 2.3]{AD}.
\begin{lemma}
  For all $0\leq p\leq\ell$, 
  \[
  \Omega^p(\A,m) \cong D^{\ell-p}(\A,m)\quad\text{and}\quad
  D^p(\A,m) \cong \Omega^{\ell-p}(\A,m).
  \]
\label{ID}
\end{lemma}
\begin{prop}[e.g., \cite{AD}]
We have the following exact sequences:
\begin{eqnarray}
0\rightarrow D^{p}(\A') 
\stackrel{\cdot \alpha_H}{\rightarrow } D^{p}(\A) 
\stackrel{\rho^H}{\rightarrow} D^{p}(\A^H) \\
0\rightarrow \Omega^{p}(\A) 
\stackrel{\cdot \alpha_H}{\rightarrow } 
\Omega^{p}(\A') 
\stackrel{\res_H}{\rightarrow} \Omega^{p}(\A^H) 
\end{eqnarray}
\label{es}
\end{prop}

\section{Ziegler's conjecture on generic cuts of free arrangements}

In this section, we prove the following.

\begin{theorem}
Conjectures \ref{Zconj} is true.
\label{Zconj2true}
\end{theorem}

In fact, we prove a more general result, Theorem~\ref{zconj1general} below.
We begin with some preparatory results.

By a standard argument (e.g., \cite[A4.2]{eisenbook}) we have:
\begin{lemma}
Let $M$ be a finitely generated, graded $S$-module. If $\pd_S M =p$, then 
$H_*^i(\widetilde{M})=0$ for $0<i <\ell-p-1$.
\label{pdvanish}
\end{lemma}
To state the next lemma, suppose $\A$ is a non-free arrangement,
$H\not\in \A$, and $\A^+:=\A\cup\set{H}$.  Suppose further that
$\omega_1,\ldots,\omega_s$ form a minimal set of generators for
$\Omega^1(\A)$, and there is some 
$\omega_0 \in \Omega^1(\A^+)_{-d} \setminus \Omega^1(\A)$ for which
\[
\Omega^1(\A^+)=\Omega^1(\A)+ S \omega_0.
\]
\begin{lemma}
Suppose that there is a relation of the form
\begin{equation}
\alpha_H \omega_0=\sum_{i=1}^s f_i \omega_i
\label{rel1}
\end{equation}
for 
some polynomials $f_i \in S$.  Then 
if $\Omega^1(\A)$ has a minimal free resolution 
\begin{equation}
0 \rightarrow F_t \rightarrow \cdots \rightarrow F_0 \rightarrow \Omega^1(\A) \rightarrow 0,
\label{eq1}
\end{equation}
then $\Omega^1(\A^+)$ has a free resolution of the form
\begin{equation}
0 \rightarrow F_t \rightarrow \cdots \rightarrow F_2 \rightarrow F_1 
\oplus S[d-1] \rightarrow F_0\oplus S[d] \rightarrow 
\Omega^1(\A^+) \rightarrow 0.
\label{eq2}
\end{equation}
In particular, $\pd_S \Omega^1(\A^+) \le \pd_S \Omega^1(\A)$. 
\label{freeresol}
\end{lemma}
\begin{proof}
Let us determine the kernel $K$ of the surjection $F_0 \oplus S[d] \rightarrow 
\Omega^1(\A^+)$ that sends the generator of $S[d]$ to $\omega_0$.
It is clear that the image of 
$F_1$ and $\alpha_H \omega_0=\sum_{i=1}^s f_i \omega_i$ are contained in $K$. 
A new relation in $K$ would be of the form 
\[
\sum_{i=1}^s g_i \omega_i=g\omega_0,
\]
for some $g$ and $g_i$'s in $S$.
Since the left-hand side is regular along $H$ and $\omega_0$ is not, 
we may replace $g$ by $g\alpha_H$. Then by \eqref{rel1}, it holds that 
\[
0=\sum_{i=1}^s g_i\omega_i-g\alpha_H \omega_0=
\sum_{i=1}^s g_i \omega_i-g \sum_{i=1}^s f_i \omega_i=
\sum_{i=1}^s (g_i-gf_i )\omega_i,
\]
which belongs to the image of $F_1$. 
It follows that
\[
F_1 \oplus S[d-1] \rightarrow F_0 \oplus S[d] \rightarrow \Omega^1(\A)
\]
is exact.  Since the restriction to $S[d-1]$ is injective, $F_2$ maps
onto the kernel, and \eqref{eq2} is exact.
\end{proof}
Recall that a hyperplane $H\in\A$ is generic with respect to $\A':=\A\setminus
\set{H}$ provided that $\codim(X\cap H)>\codim X$ for each nonzero flat
$X\in L(\A')$ (Definition~\ref{kgeneric}).  In this case, the matroid of
$\A^H$ is the complete principal truncation of that of $\A'$; in particular,
$L_k(\A')\cong L_k(\A^H)$ for all $0\leq k\leq \ell-2$.

\begin{prop}
  If $H\in\A$ is a generic hyperplane with respect to $\A'$, then
  for all $p\geq1$,
the restriction maps $D^p(\A) \rightarrow D^p(\A^H)$ and 
$\Omega^p(\A') \rightarrow \Omega^p(\A^H)$ 
%(see \cite{AD}, the equation (1.2))
are locally surjective.
That is, for any point $\pp \in \P^{\ell-1}=\Proj(S)$, 
the maps are surjective when localized at $\pp$.
\label{genericlocsurj}
\end{prop}
\begin{proof}  
  First consider the logarithmic derivation case, and fix a flat
  $X\in L(\A')$ with $\codim(X)\leq \ell-2$.
  By the genericity assumption, we may choose
  coordinates so that $H$ is given by $x_\ell=0$ and $\A'_X$ is
  an arrangement in $\spec(\K[x_1,\ldots,x_{\ell-1}])$.
  Then any $\theta \in D^p((\A^H)_X)$ has a %canonical
  lift to some $\tilde{\theta} \in D^p(\A')$, and
$\partial_{x_\ell}$ does not appear in the expression of
$\tilde{\theta}$. That is, $\rho\colon D^p(\A'_X) \rightarrow D^p((\A^H)_X)$ is
surjective, and the analogous argument works for the logarithmic differential
forms.   
\end{proof}
\begin{prop}
If $H\in \A$ has the property that the restriction
$\rho\colon D^{\ell-1}(\A) \rightarrow D^{\ell-1}(\A^H)$ is 
surjective, then there is some
$\omega_0 \in \Omega^1(\A)_{-d} \setminus \Omega^1(\A')$ for which
\[
\Omega^1(\A)=\Omega^1(\A')+ S \omega_0,
\]
where $d:=|\A|-|\A^H|$ and $\A':=\A\setminus{H}$.
\label{plus1}
\end{prop}
\begin{proof}
  Choosing coordinates appropriately, we let $\alpha_H=x_1$.
  Then $\theta_0:= Q(\A^H) \wedge_{i=2}^\ell \partial_{x_i}$ generates
  $D^{\ell-1}(\A^H)$, since $\dim_\K H=\ell-1$. By surjectivity, there
  is some $\theta \in D^{\ell-1}(\A)_{|\A^H|}$ for which 
  $\rho(\theta)=\theta_0$.
  
  Note that the coefficient $Q(\A^H)$ of $\theta_0$
  is not divisible by $x_1$ since $\rho(\theta)=\theta_0 \neq 0$.
  Thus the differential form $\omega_0 \in \Omega^1(\A)_{-|\A|+|\A^H|}=
  \Omega^1(\A)_{-1}$ dual to $\theta$ (see Lemma~\ref{ID})
  is not regular along $H$. By Lemma \ref{Tlemma}, we have
\[
\Omega^1(\A)=\Omega^1(\A')+ S \omega_0,
\]
which completes the proof.
\end{proof}

The following technical result lies at the core of our proof of Ziegler's
Conjecture~\ref{Zconj}.

\begin{theorem}
  Suppose $\A$ is an essential arrangement of rank $\ell\geq4$, and
  $H\in\A$ is generic with respect to $\A':=\A\setminus\set{H}$.
  If $\pd_S \Omega^1(\A')<\ell-2$, then
\begin{itemize}
    \item [(1)]
$\res_H\colon \Omega^1(\A') \rightarrow \Omega^1(\A^H)$ is surjective, 
\item[(2)]
a minimal set of generators for $\Omega^1(\A')$ is sent to 
a minimal set of generators 
for 
$\Omega^1(\A^H)$. Thus $\A^H$ not free, and 
\item[(3)]
there is $\omega \in \Omega^1(\A)_{-1} \setminus \Omega^1(\A')$ such that 
$\omega,\omega_1,\ldots,\omega_s$ form a minimal set of generators for 
$\Omega^1(\A)$, where $\omega_1,\ldots,\omega_s$ is a minimal set of generators
for $\Omega^1(\A')$.

\item[(4)] 
If $\A'$ is free, then
$\pd_S \Omega^1(\A)=\pd_{S/\alpha_H S} \Omega^1(\A^H)=1$.
If $\A'$ is not free, then we have
\[
1 \le \pd_{S/\alpha_H S} \Omega^1(\A^H)\le  
\pd_S \Omega^1(\A)=\pd_S \Omega^1(\A').
\]
\end{itemize}
\label{zconj1general}
\end{theorem}
\begin{proof}
We will prove (1) and then (3), followed by (4) and (2).
Let $\omega_1,\ldots,\omega_s\ (s\ge \ell) $ be a minimal set of 
generators for $\Omega^1(\A')$.  By 
Proposition \ref{genericlocsurj}, both restrictions 
$\rho\colon D^{\ell-1}(\A) \rightarrow D^{\ell-1}(\A^H)$ and
$\res_H\colon\Omega^1(\A') \rightarrow \Omega^1(\A^H)$ are locally surjective,
giving two short exact sequences
\[
0 \rightarrow \widetilde{D^{\ell-1}(\A')} 
\stackrel{\cdot \alpha_H}{\rightarrow} \widetilde{D^{\ell-1}(\A)}
\stackrel{\rho}{\rightarrow} \widetilde{D^{\ell-1}(\A^H)} \rightarrow 0
\]
and
\[
0 \rightarrow \widetilde{\Omega^1(\A)} 
\stackrel{\cdot \alpha_H}{\rightarrow} \widetilde{\Omega^1(\A')}
\stackrel{\res_H}{\rightarrow} \widetilde{\Omega^1(\A^H)} \rightarrow 0.
\]
Since $\pd_S D^{\ell-1}(\A')=\pd_S \Omega^1(\A')<\ell-2$, Lemma~\ref{pdvanish}
shows that 
\[
H_*^1(\widetilde{D^{\ell-1}(\A'))}=0.
\]
Hence Proposition \ref{gs} shows that 
\[
0 \rightarrow D^{\ell-1}(\A')
\stackrel{\cdot \alpha_H}{\rightarrow} D^{\ell-1}(\A)
\stackrel{\rho}{\rightarrow} D^{\ell-1}(\A^H) \rightarrow 0.
\]
By Proposition \ref{plus1}, there exists some
$\omega \in \Omega^1(\A)_{-1}$ for which
\[
\Omega^1(\A)=\Omega^1(\A')+S\omega.
\]
Thus Lemma \ref{freeresol} shows that
$\pd_S \Omega^1(\A') \ge \pd_S \Omega^1(\A)$ unless $\A'$ is free, and
$\pd_S\Omega^1(\A)=1$ if $\A'$ is free by Theorem \ref{dualSPOG}. 
We apply Proposition \ref{gs} and Lemma \ref{pdvanish} 
together with $\pd_S \Omega^1(\A) \le \ell-3$ to obtain that
$H_*^1(\widetilde{\Omega^1(\A))}=0$.   % changed from $\A'
Taking global sections, we obtain the exact sequence
\[
0 \rightarrow \Omega^1(\A)
\stackrel{\cdot \alpha_H}{\rightarrow} \Omega^1(\A')
\stackrel{\res_H}{\rightarrow} \Omega^1(\A^H) \rightarrow 0,
\]
proving (1). Next, let us show (3): that is, 
$\set{\omega,\omega_1,\ldots,\omega_s}$ forms a minimal set of generators for
$\Omega^1(\A)$. By construction, $\omega_i$ is regular along
$H$. Thus $\omega_1,\ldots,\omega_s$ cannot generate
$\Omega^1(\A)$. By way of contradiction, suppose, without loss, that
$\set{\omega,\omega_1,\ldots,\omega_{s-1}}$
generate $\Omega^1(\A)$. 
Let
$$
\omega_s=\sum_{i=1}^{s-1} f_i \omega_i+f \omega.
$$
Since only $\omega$ has a pole along $H$, it holds that $f=f'\alpha_H$ for some homogeneous polynomial $f'$. If 
$\deg f > 0$, then $\deg f\omega\ge 0$, contradicting $\deg \omega_s<0$. 
Hence $\deg f=0$, $\deg \omega_s=-1$, and
$\omega_s=a \omega+\sum_{i=1}^{s-1} f_i \omega_i$ for some $a \in \K$
and elements $f_i \in S$.  So $\omega_s$ has a pole along $H$, a
contradiction.  We conclude $\set{\omega,\omega_1,\ldots,\omega_s}$
forms a minimal set of generators for $\Omega^1(\A)$.

Let us prove assertion (4) regarding
$\pd_S\Omega^1(\A)$ and $\pd_S\Omega^1(\A')$.
If $\A'$ is free, we know that $\A$ is SPOG by Theorem \ref{dualSPOG}.
Now assume that $\A'$ is not free.
Continuing the same notation, suppose 
\[
\sum_{i=1}^s f_i \omega_i+\alpha_H \omega=0
\]
is a relation among $\{\omega_i\}_{i=1}^s \cup \{\omega\}$ in $\Omega^1(\A')$.
Let
\begin{equation}
0 \rightarrow F_t \stackrel{\partial_t}{\rightarrow} \cdots 
\stackrel{\partial_1}{\rightarrow}F_0 \stackrel{\partial_0}{\rightarrow} \Omega^1(\A') \rightarrow 0
\label{eq3}
\end{equation}
be a minimal free resolution of $\Omega^1(\A')$ with $s$ minimal generators
as shown above. Then we can construct a %canonical
free resolution 
\begin{equation}
0 \rightarrow F_t \stackrel{\partial_t}{\rightarrow} \cdots 
\stackrel{\partial_2}{\rightarrow}
F_1 \oplus S 
\stackrel{\partial_1 \oplus E_1}{\rightarrow}F_0 \oplus S[1] \stackrel{\partial_0\oplus E_0}{\rightarrow} \Omega^1(\A) \rightarrow 0,
\label{eq4}
\end{equation}
by letting $E_0$ send the generator of $S[1]$ to $\omega$, and defining
$E_1\colon S \rightarrow F_0\oplus S[1]$ by sending the generator to
$(f_1,\ldots,f_s,\alpha_H)$.
By Lemma~\ref{freeresol}, we see
\eqref{eq4} is a free resolution, and its minimality follows from
that of \eqref{eq3}, noting that each differential can be represented by a
matrix with entries with strictly positive degree.

Hence we have $\pd_S \Omega^1(\A)=\pd_S \Omega^1(\A')$.
We defer the remaining part of (4), regarding the projective dimension of
$\Omega^1(\A^H)$, to the end.

Next let us show (2): that is, 
$\res_H(\omega_1),\ldots,\res_H(\omega_s)$ form a minimal set of
generators for $\Omega^1(\A^H)$. Assume not, and say that
\[
\res_H(\omega_s)=\sum_{i=1}^{s-1} \res_H(f_i \omega_i).
\]
Then 
\[
\omega_s-\sum_{i=1}^{s-1} f_i \omega_i =\alpha_H \omega_0
\]
for some $\omega_0 \in \Omega^1(\A)_{(\deg \omega_s)-1}$. If $\omega_0$ is regular along $H$, then 
$\omega_0 \in \langle \omega_1,\ldots,\omega_{s-1}\rangle_S$ since $\deg \omega_0=\deg \omega_s-1$.
Hence $\omega_1,\ldots,\omega_{s-1}$ form a set of generators for $\Omega^1(\A')$, a contradiction. 
Thus $\omega_0$ is not regular along $H$. So 
by using a minimal set of generators $\omega,\omega_1,\ldots,\omega_s$ for $\Omega^1(\A)$ exhibited above, 
we may write
\[
\omega_0=g\omega+\sum_{i=1}^sg_i\omega_i
\]
for some $g$.  Since $\deg \omega=-1$, $\deg \omega_s=\deg \omega+1=0$.

However, Bath's regularity bound \cite[Theorem\ 2.29]{Bath22}
(with our grading convention) implies that each minimal generator has negative
degree, a contradiction.  Since $s \ge \ell>\rank\A^H$, the arrangement
$\A^H$ cannot be free.

Finally, we prove the last asertion from (4).  If $\A'$ is not free, then
by (2) we know that $\A^H$ is also not free.  The result follows by taking
the long exact sequence of $\Ext(-,S)$ applied to
\[
0 \rightarrow \Omega^1(\A)
\stackrel{\cdot \alpha_H}{\rightarrow} \Omega^1(\A')
\stackrel{\res_H}{\rightarrow} \Omega^1(\A^H) \rightarrow 0,
\]
noting that $\pd_{S/\alpha_H}M=(\pd_S M)-1$ for $S/\alpha_H$-modules $M$.

\end{proof}

Now we can prove a more general version of Ziegler's Conjecture \ref{Zconj} and 
Theorem \ref{Zconj3}.

\begin{theorem}
  Let $\A$ be an arrangement of rank $\ell$,
  and $X$ a generic subspace with respect to $\A$
  of dimension $k\geq 3$.  If $p:=\pd_S\Omega^1(\A)\leq k-2$, 
%  Suppose $\A$ satisfies $\pd_S \Omega^1(\A)=p<\ell-2$.
%  Let $X \not \in L(\A)$ be a generic subspace
%  with respect to $\A$. If $k := \dim_\K X \ge 3$ and $p\leq k-2$, %\le k-2$,
  then $\A\cap X$ is not free.
\label{zconjverygeneral}
\end{theorem}
\begin{proof}
%  We choose generic hyperplanes $\set{H_i}_{1\leq i\leq \ell-k}$
%  with the property that $X=H_1 \cap \cdots \cap H_{\ell-k}$, and let
%  $\A^+=\A\cup\set{H_1,\ldots,H_{\ell-k}}$, arguing by induction on
%  $\ell-k=\codim X$.

  We argue by induction on $\codim X=\ell-k$ and prove the more precise
  statement that $1\leq \pd_S\Omega^1(\A\cap X)\leq p$.
  If $\ell-k=1$, $X$ is a generic hyperplane and $\ell\geq 4$.  By
  hypothesis, $p\leq (\dim X)-2=\ell-3$, so
  this inequality follows directly from Theorem~\ref{zconj1general}(4).

  If $\ell-k\geq2$, we choose generic hyperplanes
  $\set{H_i}_{1\leq i\leq \ell-k}$ with the property that
  $X=H_1 \cap \cdots \cap H_{\ell-k}$.  Let $Y=H_1\cap \cdots H_{\ell-k-1}$,
  and $\B=\A\cap Y$, an arrangement of rank $k+1$.  By induction,
  \[
  \pd_S\Omega^1(\B)\leq p\leq k-2=\rank\B-3,
  \]
  so we may apply Theorem~\ref{zconj1general}(4)
  again to conclude $\B\cap H_{\ell-k}=\A\cap X$ is not free, and
  moreover s $\pd_S(\Omega^1(\A\cap X)\leq p$.
\end{proof}
\begin{proof}[Proof of Theorem \ref{Zconj2true}]
  By Theorem \ref{zconj1general}, the minimal number of generators of
  $\Omega^1(\A')$ is unchanged by restriction, and Ziegler's conjectures follow
  from Theorem~\ref{zconjverygeneral}.
\end{proof}

\begin{rem}
The proofs above do not apply when $\ell=3$ since 
genericity does not confirm the surjectivity of $\res_H$ and $\rho$.
\end{rem}
The genericity assumption is indispensible: if
$X$ is not $(\ell-1)$-generic, then the conclusion of Theorem \ref{Zconj3}
does not hold.

\begin{example}
Let 
\[
Q(\A)=\prod_{i=1}^4 x_i \prod_{i=1}^3 (x_i^2-x_4^2)(x_i^2-4x_4^2)
\prod_{i=2}^3 (x_i^2-9x_4^2)
(x_3^2-16x_4^2),
\]
a product of $22$ linear forms.  
Then $\A$ is free with $\exp(\A)=(1,5,7,9)$. Let 
$H$ be given by $x_1+x_2+x_3=0$ and let 
$\A^+:=\A \cup\{H\}$. Then 
$\A^{+H}$ is free with exponents 
$(1,10,11)$. Clearly $H$ is $2$-generic but not $3$-generic with respect to
$\A$. So this does not contradict Theorem \ref{Zconj2true}.
\end{example}

\begin{example}
Consider the arrangement $\A$ of $9$ hyperplanes in $\CC^4$ given by the columns of the matrix 
\[
\begin{pmatrix}
1 &0& 0& 1& 0& 0& 1& 1& 0\\
0& 1& 0& 0& 1& 0& 1& 0& 1\\
0& 0& 1& 0& 0& 1& 0& 1& 1\\
0& 0& 0& 1& 1& 1& 1& 1& 1\\
\end{pmatrix}
\]
and let $H$ be the generic hyperplane defined by the vector
$(1,3,5,7)$.  A Macaulay2~\cite{GS} calculation shows that
$\Omega^1(\A)$ and $\Omega^1(\A\cap H)$ are generated in degrees
$\set{-1,-2}$ and $\set{-1,-2,-3}$, respectively.  This shows the
restriction
\[
\res_H\colon \Omega^1(\A)\to \Omega^1(\A\cap H)
\]
is not surjective, so we see that the hypothesis in Theorem \ref{zconjverygeneral} that  $\pd\Omega^1(\A)\le \dim X-2=1$ is, in fact, necessary.
\end{example}

\section{Counterexamples to Ziegler's conjecture on criticality}

We now examine Conjecture \ref{Zconj2} and show it does not
hold in general.  The following basic observation from \cite[\S8]{Z2}
motivates the conjecture:
\begin{prop}
  Suppose $\omega\in\Omega^1(\A)_{-k}$
  but $\omega\not\in\Omega^1(\A\setminus\set{H})$ for
  some $H\in \A$.  Then $\abs{\A}-\abs{\A^H}\geq k$.
\end{prop}
\begin{proof}
  Since $\omega\not\in\Omega^1(\A\setminus\set{H})$, it must have a pole
  on $H$.  In the notation of Lemma~\ref{Tlemma}, the coefficient $f_1\neq 0$,
  so $\deg(f_1)\geq\abs{\A^H}$ and
  \[
  \deg(\omega)\geq \abs{\A^H}-\abs{\A}.
  \]
\end{proof}
It follows that, if $\A$ is $k$-critical with $\omega\in\Omega^1(\A)_{-k}$,
then $\omega$ has poles on all hyperplanes and
$k\leq\abs{\A}-\abs{\A^H}$ for all $H$.  To find arrangements for which
no hyperplane achieves equality, we will use the
following relationship between criticality and
classical deletion theorems.

\begin{theorem}
\label{thm:Zconj2False}
  Let $\A$ be free and irreducible
  with $\exp(\A)=(1,a,b)_\le$. Suppose
  that $\A \setminus \set{H}$ is not free for any $H\in\A$.
  Then $\A$ is $b$-critical, and $\A$
  is a counterexample to Conjecture \ref{Zconj2}. 
\label{cegeneral}
\end{theorem}

\begin{proof}
  Let us first check that, if $\A$ is $b$-critical, then it provides a
  counterexample.
By Theorem \ref{A}, either $\abs{\A^H}=b+1$ or $\abs{\A^H} \le a+1$. Since $\A 
\setminus \set{H}$ is not free, Terao's Deletion 
Theorem (Theorem \ref{adddel}) shows that we must have $\abs{\A^H}<a+1$.
That means
\begin{eqnarray*}
  \abs{\A}-\abs{\A^H}&>&1+a+b-(a+1)\\
  &=&b,
\end{eqnarray*}
for all $H\in \A$.  For flats $L$ of codimension $2$ we have
$\abs{\A}-\abs{\A^L}=a+b-1>b$ as well, since $a>1$ for irreducible free
arrangements.

To check that $\A$ is indeed $b$-critical, let $\omega_1$, $\omega_2$,
$\omega_3$ be forms generating $\Omega^1(\A)$ of degrees
$1$, $a$, $b$, respectively.  Suppose instead that $\A$ contains a
hyperplane $H$ for which $\Omega^1(\A')_{-b}\neq 0$, where
$\A':=\A\setminus\set{H}$.  Then $\omega_3\in \Omega^1(\A')$.
Clearly $\omega_1, \alpha_H\omega_2\in \Omega^1(\A')$ as well.  By
Saito's Criterion~\cite[Theorem\ 1.8]{Sa}, these forms generate $\Omega^1(\A')$,
which implies $\A'$ is free, a contradiction.
\end{proof}

\begin{cor}
The reflection arrangement $G(r,r,3)$ for $r\geq3$, defined by the equation
\[
Q(\A):=(x^r-y^r)(y^r-z^r)(x^r-z^r)
\]
in $\CC^3$, is $(2r-2)$-critial, but there are no $H \in \A$ such 
that $\abs{\A}-\abs{\A^H}=2r-2$.
That is, Conjecture \ref{Zconj2} does not hold for $\A$.
\label{ce}
\end{cor}
\begin{proof}
  The reflection arrangements $G(r,r,3)$ have exponents $(1,r+1,2r-2)$ by
  \cite[Corollary 6.86]{OT}.
  Note that 
$\abs{\A^H}=r+1$ for each $H \in \A$.
  By the Deletion Theorem, then, the deletion $\A\setminus\set{H}$ is
  not free, and we may use Theorem~\ref{thm:Zconj2False}.
\end{proof}

Other known arrangements provide more counterexamples.
\begin{example}~
  \begin{enumerate}
\item
Let $\A$ be the arrangement in $\R^3$ 
whose deconing consists of all edges and diagonals of the 
regular pentagon. It is known that $\A$ is free with $\exp(\A)=(1,5,5)$
(see \cite[Ex.\ 4.59]{OT}).  It follows from the Addition-Deletion Theorem
that no deletion is free.  
Theorem \ref{thm:Zconj2False} shows that $\A$ is $5$-critical, and
a counterexample to Conjecture \ref{Zconj2}. 
\item
Two non-recursively free arrangements $\A_{13}$ and $\A_{15}$ in $\CC^3$ were
found in \cite[\S6,\,\S7]{ACKN}. They have $\abs{\A_{13}}=13$ and
$\abs{\A_{15}}=15$.  Both are free with 
$\exp(\A_{13})=(1,6,6)$ and $\exp(\A_{15})=(1,7,7)$, while
$\abs{\A^H_{13}}=6$ and $\abs{\A^H_{13}}=7$ for each respective hyperplane.
The cardinalities of the restrictions imply that no deletion is free, so
they are also counterexamples to Conjecture \ref{Zconj2}.
\end{enumerate}
\label{pentagon}
\end{example}

We remark by way of conclusion that
these ideas can be refined in the following way. We omit the proof, since
it is similar to \cite[Theorem~1.9]{A5}.

\begin{theorem}
Let $\A$ be free with basis $\omega_1,\ldots,\omega_\ell$ for $\Omega^1(\A)$ and let $H 
\in \A$. If $\A':=\A \setminus \{H\}$ is not free, and $\omega_i \in \Omega^1(\A')$ for $1 \le i \le \ell-2$, then 
$\A'$ is dual SPOG with exponents $(-1,-d_2,\ldots,-d_{\ell-2},-d_{\ell-1}+1,
-d_\ell+1)$ and level $d=d_{\ell-1}+d_\ell- |\A|+|\A^H|$.
\label{omegaSPOGminus}
\end{theorem}


\begin{thebibliography}{ABCHT}

\bibitem{A}
T. Abe, 
Roots of characteristic polynomials and 
and intersection points of line arrangements. 
\textit{J. Singularities}, 
\textbf{8} (2014), 100--117.

\bibitem{A2}
T. Abe,
Divisionally free arrangements of hyperplanes. 
\textit{Invent. Math.} \textbf{204} (2016), no. 1, 317--346.

\bibitem{A4}
T. Abe, 
Deletion theorem and combinatorics of hyperplane arrangements.
\textit{Math. Ann}.
\textbf{373} (2019), issue 1--2, 581--595. 


\bibitem{A5}
T. Abe, 
Plus-one generated and next to free arrangements of hyperplanes. 
\textit{Int. Math. Res. Not.} 
\textbf{2021}, no. 12, 9233-–9261.

\bibitem{A6}
T. Abe, 
Addition-deletion theorem for free hyperplane arrangements and 
combinatorics. 
\textit{J. Algebra} \textbf{610} (2022), 1--17. 

\bibitem{A9}
T. Abe, 
Projective dimensions of hyperplane arrangements.
arXiv:2009.04101 (2020).

\bibitem{A10}
T. Abe, 
Generalization of the addition and restriction theorems from free arrangements to the class of projective dimension one. 
arXiv:2206.15059 (2022). 

\bibitem{ACKN}
T. Abe, M. Cuntz, H. Kawanoue and T. Nozawa, 
Non-recursive freeness and non-rigidity of plane arrangements. 
\textit{Discrete Math}. \textbf{339} (2016), no. 5, 1430–-1449.


\bibitem{AD}
T. Abe and G. Denham, 
Deletion--restriction for logarithmic forms on multiarrangements,
arXiv:2203.04816 (2022).

\bibitem{ATW2}
T. Abe, H. Terao and M. Wakefield, 
The Euler multiplicity and addition-deletion theorems for multiarrangements. \textit{J. London Math. Soc}.,
\textbf{77} (2008), no. 2, 335--348.

\bibitem{Bath22}
D.\ Bath,
Hyperplane arrangements satisfy (un)twisted logarithmic comparison theorems, applications to ${\mathcal D}_X$-modules,
arXiv:2202.01462.

\bibitem{eisenbook}
David Eisenbud, \emph{Commutative algebra}, Graduate Texts in Mathematics, vol.
  150, Springer-Verlag, New York, 1995, With a view toward algebraic geometry.
  %MR{97a:13001}

\bibitem{GS}
D. R. Grayson and M. E. Stillman. Macaulay2, a software system for research in
algebraic geometry. Available at http://www.math.uiuc.edu/Macaulay2/.

\bibitem{OT} P. Orlik and H. Terao, \textit{Arrangements of hyperplanes}.
Grundlehren der Mathematischen Wissenschaften, 
\textbf{300}. Springer-Verlag, Berlin, 1992.

\bibitem{Sa}
K. Saito, 
Theory of logarithmic differential forms and logarithmic vector fields.
\textit{J. Fac. Sci. Univ. Tokyo} \textbf{27} (1980), 265--291.   


\bibitem{T1}
H. Terao, 
Arrangements of hyperplanes and their freeness I, II. 
\textit{J. Fac. Sci. Univ. Tokyo} \textbf{27} (1980), 293--320.   

\bibitem{Y1}
M. Yoshinaga,
Characterization of a free arrangement and
conjecture of
Edelman and Reiner. \textit{Invent. Math.} \textbf{157} (2004), no. 2,
449--454.

\bibitem{Yu}
S. Yuzvinsky, 
The first two obstructions to the freeness of arrangements.
\textit{Trans. Amer. Math. Soc}. \textbf{335} (1993), no. 1, 231--244.

\bibitem{Z}
G. M. Ziegler, 
Multiarrangements of hyperplanes and their freeness.  Singularities (Iowa City, IA, 1986),  345--359,
Contemp. Math., {\bf 90}, Amer. Math. Soc., Providence, RI, 1989. 


\bibitem{Z2}
G. M. Ziegler, Combinatorial construction of logarithmic differential forms.
\textit{Adv. Math}. \textbf{76} (1989), 116--154.
	
\end{thebibliography}
\end{document}